\documentclass[a4paper,10pt]{article}

\usepackage[utf8]{inputenc} %encodage des caracteres
\usepackage[T1]{fontenc} %fontes
\usepackage[french,english]{babel} %typographie
\usepackage{lmodern}
\usepackage{amsfonts,amsmath,amssymb,amsthm} %packages ams utiles pour ecrire les maths
\usepackage{mathrsfs} %autre police pour ecrire math
\usepackage{supertabular} %pour les très grands tableaux
\usepackage{tikz} %faire des dessins, figures, tracer fonctions ...
	\usetikzlibrary{calc,arrows,decorations}
\usepackage{multirow}
\usepackage{float}

\newtheorem{Thm}{Theorem}[section]

\newtheorem{remark}[Thm]{Remark}

\DeclareMathOperator{\sech}{sech}

\newcommand\R{\mathbb{R}}

\newcommand\N{\mathbb{N}}
\newcommand\Z{\mathbb{Z}}

\usepackage{geometry}
\title{Numerical solution of the generalized 
Kadomtsev-Petviashvili equations with compact finite difference schemes}
\author{ {\bf J.-P. Chehab, P. Garnier$^*$, Y. Mammeri}\\ 
{\small Laboratoire Ami\'enois de Math\'ematique Fondamentale et Appliqu\'ee,} \\  {\small CNRS UMR 7352, Universit\'e 
de Picardie Jules Verne, France}\\
{\small 80039 Amiens, France.}\\
{\small $^*$Corresponding author: pierre.garnier@u-picardie.fr}
}
\date{2016}

\begin{document}

\maketitle

\begin{center}
\begin{minipage}[t]{13.5cm}{\footnotesize {\bf Abstract.} 
We propose compact finite difference schemes to solve the KP equations $u_t + u_{xxx} + u^p u_x + \lambda \partial_x^{-1} u_{yy} = 0$. When $p=1$, this equation describes the propagation of small amplitude long waves in shallow water with weak transverse effects. We first present the numerical schemes which are compared to the Fourier spectral method. After establishing the numerical convergence, the scheme is validated. We then depict the behavior of solutions in the context of solitons instabilities and the blow-up.} 
\end{minipage}  
\end{center}

\begin{minipage}[t]{13.5cm}{\footnotesize {\bf Keywords.} 
Kadomtsev-Petviasvili equations, compact finite difference schemes, blow-up, transverse instabilities, Zaitsev solitons, line-soliton.
}
\end{minipage}

\vspace{0.5cm}

{\footnotesize {\bf MS Codes.} 35Q53, 65M06, 65M70, 35G25, 76B15 }

% \noindent\rule{\linewidth}{.5pt}

% \tableofcontents

% \noindent\rule{\linewidth}{.5pt}

% \graphicspath{{pictures/}}

\section{Introduction}
	The Kadomtsev–Petviashvili equations
$$
	u_t + u_{xxx} + u u_x \pm \partial_x^{-1} u_{yy} = 0
$$
were proposed to describe
the propagation of long dispersive waves in shallow water taking into account weakly transverse effect in the 
$y-$direction \cite{KP70,La03,LS06}.
Here $\partial_x^{-1}$ denotes the antiderivative and is defined by the Fourier transform
$$
\widehat{\partial_x^{-1}u}\left( \xi_1, \xi_2 \right) = \frac{\hat u\left( \xi_1, \xi_2 \right)}{i\xi_1}.
$$
This antiderivative is introduced from the one-dimensional transport operator $\partial_t + \partial_x$ by 
considering weakly transverse perturbations, {\it i.e.} in the frequency region $|\xi_2/\xi_1|\ll 1.$ 

In this paper, we are concerned with the generalized Kadomtsev-Petviashvili equations
\begin{equation}
	u_t + u_{xxx} + u^p u_x + \lambda \partial_x^{-1} u_{yy} = 0,  
  \label{eq1}
\end{equation}
with $p \geq 1.$ Here the equation is called KP-I when $\lambda=-1$ and KP-II when $\lambda=1$. 

It is well known that sufficiently smooth solutions vanishing at infinity  preserve 
the following mass
$$
\int_{-\infty}^{+\infty}  \int_{-\infty}^{+\infty}  u^2(x,y,t) dx dy
$$
and energy
$$
\int_{-\infty}^{+\infty}  \int_{-\infty}^{+\infty} \left( \frac{u^{p+2}}{(p+1)(p+2)} -\frac{u_x^2}{2}
+ \frac{\lambda(\partial_x^{-1}u_y)^2}{2}  \right)(x,y,t) dx dy.
$$
Note also that one has a mass-zero constraint on $x$ due to the equation itself, 
more precisely we have
$$
\int_{-\infty}^{+\infty} u(x,y,t) dx =0,\; y\in \R, t \neq 0,
$$
without any restraint on the initial data \cite{MST07}.
This condition is of paramount importance in the construction of numerical schemes.

Numerical behavior of the KP equations has been intensely studied recently \cite{Ha01, HMM09, ISS95, ISR02, KSM07, KS10, KR10, My07, Wa01}.
Due to the definition of the antiderivative, spectral methods has been privileged.
This needs to impose periodic boundary conditions.
In our work, compact finite difference schemes are proposed \cite{Lele}. Their benefits are
to easily adapt  to all boundary conditions, and also to increase the order. Remark that, we
can prove \cite{My09}
that  the antiderivative in a bounded domain is equivalently given by an integration in space
$$
\partial_x^{-1} u (x,y,t) = \int_{-\infty}^x u(x',y,t) dx'.
$$
The paper is organized as follows. Section $2$ deals with the numerical schemes. It consists
of building compact finite difference and spectral-compact schemes of arbitrary large order. 
These new schemes are validated in Section $3$ with the Zaitsev soliton \cite{Za83}
\[
\psi_c(x,y,t) = 12 \alpha^2 \frac{1-\beta \cosh(\alpha x - \omega t)
  \cos(\delta y)}{(\cosh(\alpha x -\omega t) -\beta \cos(\delta y))^2}\,, 
\beta = \sqrt{\frac{\delta^2-3\alpha^4}{\delta^2}}\,, 
3\alpha^4 < \delta^2\,, \omega = \frac{\delta^2 + \alpha^4}{\alpha}, \; c = \frac{\omega}{\alpha}.
\]
as the exact solution of the KP-I equation. Finally, the robustness of the compact scheme is tested through three classical experiments and is compared to fully  spectral method. First,
the instability of the Zaitsev soliton is established using Schwartzian perturbations \cite{KS10}.
Then transverse instabilities of the Korteweg-de Vries line-soliton 
\[
\Phi_c(x,t) = \left( \frac{(p+1)(p+2)}{2}\,c\right)^{1/p}{\rm sech}^{2/p}
\left(\frac{p\sqrt{c}}{2}\, (x-ct) \right),
\]
under the 
KP-I flow is studied \cite{APS97,AS97,RT08,RT09,RT10,RT11,RT12}. 
We conclude with the blow-up in finite time \cite{Li01,Li02,Sa93,Sa95} 
of the solution starting from the initial data
$$
u_0(x) = 3 \partial_{xx} {\rm e}^{-\alpha(x^2+y^2)}.
$$

\section{Time and space discretization of KP equations}
\label{sec:num_scheme}
	The numerical schemes are introduced. We focus on the compact schemes. Then we remind the spectral scheme used in \cite{Ha01, HMM09, ISS95, KSM07, KS10}.

\subsection{The time discretization}

Let $\Delta t>0$ be the time step. In order to preserve the mean and the $L^2$-norm, we use a Sanz-Serna scheme. It is written
\begin{equation}\label{time_scheme}
	\frac{u(t_{n+1}) - u(t_n)}{\Delta t} + \partial_{xxx} \left( \frac{u(t_n) + u(t_{n+1})}{2} \right) + \frac{1}{p+1} \partial_x  \left( \frac{u(t_n) + u(t_{n+1})}{2} \right)^{p+1} + \lambda \partial_x^{-1} \partial_{yy} \left( \frac{u(t_n) + u(t_{n+1})}{2} \right) = 0.
\end{equation}

\begin{Thm}
	For solutions smooth enough, the scheme \eqref{time_scheme} conserves the mean and the $L^2$-norm.
\end{Thm}

\begin{proof}
	The mass conservation is showed by integrating \eqref{time_scheme} over space. The $L^2$-norm conservation is obtained by multiplying \eqref{time_scheme} with ${\displaystyle  \frac{u(t_n) + u(t_{n+1})}{2}}$, then integrating over space.
\end{proof}

\subsection{The space discretization}

From \eqref{time_scheme}, we write the discretization in space as
\begin{equation}\label{full_gnl_scheme}
	\frac{U_{n+1} - U_n}{\Delta t} + A_{x,3} \left( \frac{U_{n+1} + U_n}{2} \right) + \frac{1}{p+1} A_{x,1}  \left( \frac{U_{n+1} + U_n}{2} \right)^{p+1} + \lambda A_{x,-1} A_{y,2} \left( \frac{U_{n+1} + U_n}{2} \right) = 0,
\end{equation}
where $A_{x,k}$, $A_{y,k}$, is the matrix associated to the discretization of the $k$-th derivative in the $x$ direction and the $y$ direction, respectively.  Here, $U_n$ denotes the approximation of $u$ at time $t_n=n\Delta t$, $n \in \N$. We now mention different ways to compute these matrices, namely through compact schemes and Fourier transform.

\subsubsection{Compact finite different schemes}

We describe how to discretize the derivative and the antiderivative operators with high-order schemes.
To simplify the readings, and, since two dimensional discretizations are easily obtained from the Kronecker product, we focus on the one dimensional case. 
We just remind that the Kronecker product of two matrices is the tensor product and has the same notation $\otimes$.
Let $f$ be a function defined in $[-L,L]$, $L>0$, and let $h>0$ denote the space step. We set 
 $x_i = -L + h(i-1)$  and $f_i = f(x_i)$ for $0 \leq i \leq N$. 
Let us remind the principle of the compact finite difference method. The classical order $2$ central scheme that discretizes the first derivative, is given by
\[ f_i' = \frac{f_{i+1}-f_{i-1}}{2h}. \]
To improve the order, we consider more points \cite{Lele}. For example, the scheme
\[ \alpha \left( f_{i-1}' + f_{i+1}' \right) + f_i' = b \frac{f_{i-2} - f_{i+2}}{4h} + a \frac{f_{i-1} - f_{i+1}}{2h}. \]
allows to reach order 6 by choosing suitable coefficients $\alpha$, $a$ and $b$ from Taylor expansions. Indeed, the only possible values are $a=\frac{14}{9}$, $\alpha=\frac{1}{3}$ and $b=\frac{1}{9}$. To obtain order 2, the parameters has to verify the condition $1+2\alpha = a+b$ e.g. $a=1$, $\alpha=b=0$ gives the classical central scheme. A 4-th order can be provided by $a=\frac{3}{2}$, $\alpha=\frac{1}{4}$ and $b=0$. We rewrite these schemes as a system $P_m F' = Q_m F$ where $P_m$ and $Q_m$ are the matrices defined by

\[ P_m = 
	\begin{pmatrix}
   		\shorthandoff{:}\begin{tikzpicture}[scale=0.75]
   			% Coin haut gauche
   			\node(1) at (0,0) {1};
   			\node(alpha1) at (0.75,0) {$\alpha$};
   			\node(alpha2) at (0,-0.75) {$\alpha$};
   			% Coin bas droit
   			\node(1bis) at (5,-5) {1};
   			\node(alpha3) at (4.25,-5) {$\alpha$};
   			\node(alpha4) at (5,-4.25) {$\alpha$};
   			% Coin haut droit
   			\node at (5,0) {$\alpha$}; 
   			% Coin bas gauche
   			\node at (0,-5) {$\alpha$};
   			% Relier en pointillées les alpha
   			\draw[dotted] (1)--(1bis);
   			\draw[dotted] (alpha1)--(alpha4);
   			\draw[dotted] (alpha2)--(alpha3);
   			% Les zeros
   			\node at (4,-1.25) {(0)};
   			\node at (1,-3.75) {(0)};
   		\end{tikzpicture}\shorthandon{:}
	\end{pmatrix}, Q_m = 
	\frac{1}{2h} \begin{pmatrix}
   		\shorthandoff{:}\begin{tikzpicture}[scale=0.75]
   			% Coin haut gauche
   			\node(0) at (0,0) {0};
   			\node(a1) at (0.75,0) {$a$};
   			\node(b1) at (1.5,0) {$\frac{b}{2}$};
   			\node(a2) at (0,-0.75) {$-a$};
   			\node(b2) at (0,-1.5) {$-\frac{b}{2}$};
   			% Coin bas droit
   			\node(0bis) at (7.5,-7.5) {0};
   			\node(a3) at (7.5,-6.75) {$a$};
   			\node(b3) at (7.5,-6) {$\frac{b}{2}$};
   			\node(a4) at (6.75,-7.5) {$-a$};
   			\node(b4) at (6,-7.5) {$-\frac{b}{2}$};
   			% Coin haut droit
   			\node at (7.5,0) {$-a$};
   			\node at (6.75,0) {$-\frac{b}{2}$};
   			\node at (7.5,-0.75) {$-\frac{b}{2}$};
   			% Coin bas gauche
   			\node at (0,-7.5) {$a$};
   			\node at (0.75,-7.5) {$\frac{b}{2}$};
   			\node at (0,-6.75) {$\frac{b}{2}$};
   			% Relier en pointillées les alpha, beta
   			\draw[dotted] (0.south east)--(0bis.north west);
   			\draw[dotted] (a1.south east)--(a3.north west);
   			\draw[dotted] (a2.south east)--(a4.north west);
   			\draw[dotted] (b1.south east)--(b3.north west);
   			\draw[dotted] (b2.south east)--(b4.north west);
   			%Les zeros
   			\node at (5.75,-2) {(0)};
   			\node at (1.75,-5.5) {(0)};
   		\end{tikzpicture}\shorthandon{:}
   	\end{pmatrix}, \]
$F$ and $F'$ are vectors associated to the function and its derivative.
 
Then the matrix associated to the first derivative is $P_m^{-1} Q_m$. We emphasize that the coefficients in the upper right corner and the lower left corner correspond to the periodic bound conditions. Concerning other boundary conditions, which do not allow self centered scheme, we rather take an offcenter scheme. An admissible scheme for the first derivative is
\[ f_i' + \alpha f_{i+1}' = \frac{1}{h} \left( a f_i + b f_{i+1} + c f_{i+2} + d f_{i+3} \right), \]
allows to reach the fourth order.

In the same way, we write for the second derivative
\begin{equation*} 
	\alpha \left( f_{i-1}'' + f_{i+1}'' \right) + f_i'' = b \left( \frac{f_{i+2} - 2f_i + f_{i-2}}{4h^2} \right) + a \left( \frac{f_{i+1} - 2f_i + f_{i-1}}{h^2} \right),
\end{equation*}	
and the Taylor expansions provide
\begin{itemize}
	\item order 2  
		\[ 1+2\alpha = a+b. \]
	\item order $m \in \{ 4, 6 \}$
		\begin{align*}
			& 1+2\alpha = a+b, \\
			& \frac{2\alpha}{(2l)!} = 2 \frac{4^{l}b+a}{(2l+2)!}, \quad l=1 \ldots \frac{m}{2}-1.
		\end{align*}
\end{itemize}

The following scheme to discretize the third derivative
\begin{equation*} 
	\alpha \left( f_{i-1}^{(3)}  + f_{i+1}^{(3)}  \right) + f_i^{(3)} = a \left( \frac{f_{i+2} - f_{i-2} -2(f_{i+1}-f_{i-1})}{2h^3} \right) + b \left( \frac{f_{i+3} - f_{i-3} -3(f_{i+1}-f_{i-1})}{8h^3} \right)	
\end{equation*}
gives
\begin{itemize}
	\item an order 2 scheme if
		\[ 1+2\alpha+2\beta = a+b+c. \]
	\item and an order $m \in \{ 4,6 \}$ if
		\begin{align*}
			& 1+2\alpha = a+b, \\
			& \frac{2\alpha}{(2l)!} = \left[ 2a \left( 2^{2(l+1)} - 1 \right) + \frac{3}{4}b \left( 3^{2(l+1)} - 1 \right) \right] \frac{1}{(2l+3)!}, \quad l=1 \ldots \frac{m}{2}-1.
		\end{align*}
\end{itemize}		

It remains to compute the discretization of the antiderivative $\partial_x^{-1}$, which is not immediate. The matrix $P_m^{-1} Q_m$ associated with the operator $\partial_x$ as above beeing not invertible in general. More precisely, we have $P_m F' = Q_m F$ where $F$ and $Q_m$ is not invertible. To close the system, we use the mass zero constraint of the KP equation \cite{MST07}. Hence, we impose $\sum F_i = 0$ and $\sum F_i' = 0$. Then the last line of $Q_m$ and $P_m$ are replaced with 1 to obtain two new matrices $\overline{Q_m}$ and $\overline{P_m}$

\[ \overline{P_m} = 
	\begin{pmatrix}
   		\shorthandoff{:}\begin{tikzpicture}[scale=0.75]
   			% Coin haut gauche
   			\node(1) at (0,0) {1};
   			\node(2) at (0,-0.5) {$\alpha$};
   			\node(3) at (0.5,0) {$\alpha$};
   			% Coin bas droit
   			\node(1bis) at (4.5,-4.5) {1};
   			\node(3bis) at (5,-4.5) {$\alpha$};
   			\node(2bis) at (4,-4.5) {$\alpha$};
   			%Ligne de 1
   			\node(l1) at (0,-5) {1};
   			\node(l2) at (5,-5) {1};
   			% Relier en pointillées les 1
   			\draw[dotted] (1)--(1bis);
   			\draw[dotted] (2)--(2bis);
   			\draw[dotted] (3)--(3bis);
   			\draw[dotted] (l1)--(l2);
   			% Les zeros
   			\node at (4,-1.25) {(0)};
   			\node at (1,-3.75) {(0)};
   		\end{tikzpicture}\shorthandon{:}
	\end{pmatrix}, \overline{Q_m} = 
	\frac{1}{2h} \begin{pmatrix}
   		\shorthandoff{:}\begin{tikzpicture}[scale=0.75]
   			% Coin haut gauche
   			\node(0) at (0,0) {0};
   			\node(a1) at (0.75,0) {$a$};
   			\node(b1) at (1.5,0) {$\frac{b}{2}$};
   			\node(a2) at (0,-0.75) {$-a$};
   			\node(b2) at (0,-1.5) {$-\frac{b}{2}$};
   			% Coin bas droit
   			\node(0bis) at (6.75,-6.75) {0};
   			\node(a3) at (7.5,-6.75) {$a$};
   			\node(b3) at (7.5,-6) {$\frac{b}{2}$};
   			\node(a4) at (6,-6.75) {$-a$};
   			\node(b4) at (5.25,-6.75) {$-\frac{b}{2}$};
   			% Coin haut droit
   			\node at (7.5,0) {$-a$};
   			\node at (6.75,0) {$-\frac{b}{2}$};
   			\node at (7.5,-0.75) {$-\frac{b}{2}$};
   			% Coin bas gauche
   			\node at (0,-6.75) {$\frac{b}{2}$};
   			%Ligne de 1
   			\node(l1) at (0,-7.5) {1};
   			\node(l2) at (7.5,-7.5) {1};
   			% Relier en pointillées les alpha, beta
   			\draw[dotted] (0.south east)--(0bis.north west);
   			\draw[dotted] (a1.south east)--(a3.north west);
   			\draw[dotted] (a2.south east)--(a4.north west);
   			\draw[dotted] (b1.south east)--(b3.north west);
   			\draw[dotted] (b2.south east)--(b4.north west);
   			\draw[dotted] (l1)--(l2);
   			%Les zeros
   			\node at (5.75,-2) {(0)};
   			\node at (1.75,-5.5) {(0)};
   		\end{tikzpicture}\shorthandon{:}
   	\end{pmatrix}. \]
   	
This new matrix $\overline{Q_m}$ becomes invertible (when $N$ is odd) and 
\[ F = \left( \overline{Q_m} \right)^{-1} \overline{P_m} F', \quad \left( \overline{D_x} \right)^{-1} = \left( \overline{Q_m} \right)^{-1} \overline{P_m}. \]

\begin{remark}
	If $N$ is even, if $C_i$ denote the columns of $\overline{Q_m}$, the operation $C_0-C_1+C_2-C_3 \cdots$ is null and $\overline{Q_m}$ is not invertible.
\end{remark} 

Gathering $D_x$, $D_{yy}$, $D_{xxx}$ and $\left( \overline{D_x} \right)^{-1}$, the discretizations of $\partial_x$, $\partial_{yy}$, $\partial_{xxx}$ and $\partial_x^{-1}$ respectively, we obtain the nonlinear system

\begin{multline}\label{schema_SS_SC}
	\frac{U_{n+1}-U_{n}}{\Delta t} + I_y \otimes D_{xxx} \frac{U_{n+1}+U_{n}}{2} +\\
	\frac{1}{p+1} I_y \otimes D_x \left( \left( \frac{U_{n+1}+U_{n}}{2} \right)^{p+1} \right) + \lambda\left( I_y \otimes \left( \overline{D_x} \right)^{-1} \right) \left( D_{yy} \otimes I_x \right) \left( \frac{U_{n+1}+U_{n}}{2} \right) = 0.
\end{multline}
It can be rewritten
\begin{multline*}
	\left[ I_y \otimes \left( I_x + \frac{\Delta t}{2} D_{xxx} \right) + \lambda \frac{\Delta t}{2} D_{yy} \otimes \left( \overline{D_x} \right)^{-1} \right] U_{n+1} = \\
	\left[ I_y \otimes \left( I_x - \frac{\Delta t}{2} D_{xxx} \right) - \lambda \frac{\Delta t}{2} D_{yy} \otimes \left( \overline{D_x} \right)^{-1} \right] U_n - \frac{\Delta t}{p+1} I_y \otimes D_x \left( \left( \frac{U_{n+1}+U_n}{2} \right)^{p+1} \right). 
\end{multline*}	

A Picard fixed-point method is used at each time iteration to compute $U_{n+1}$.

\begin{Thm}
	The scheme \eqref{schema_SS_SC} preserves the mean and the mass zero constraint.
\end{Thm}

\begin{proof}
	We verify that ${\displaystyle \sum_{i} \left( U_{n+1} \right)_i = \sum_{i} \left( U_{n} \right)_i }$. We write the 1D discretization matrices as
	 \[ D_x = P_1^{-1} Q_1,\ D_{yy} = P_2^{-1} Q_2,\ D_{xxx} = P_3^{-1} Q_3 \mbox{~and~} \left( \overline{D_x} \right)^{-1} =\left( \overline{Q_1} \right)^{-1} \overline{P_1}.\] 
	
	First ${\displaystyle \sum_i \left( I_y \otimes D_{xxx} \frac{U_{n+1}+U_{n}}{2} \right)_i = 0 }$. Indeed, the matrix $Q_3$ is antisymmetric, $P_3 $ is symmetric and $\forall j$
	\[ \sum_i \left( Q_3 \right)_{i,j} = 0, \ \sum_i \left( P_3 \right)_{i,j} = C, \quad C \neq 0. \]
	Let us verify that ${\displaystyle \sum_i \left( D_{xxx} X \right)_i = 0 }$ for any vector $X$. We set ${\displaystyle Y = D_{xxx} X = P_3^{-1} Q_3 X }$, thus $P_3 Y = Q_3 X$.
	We have
	\[ \sum_i \left( Q_3 X \right)_i = \sum_i \sum_k \left[ \left( Q_3 \right)_{i,k} X_k \right] = \sum_k \left[ \sum_i \left( Q_3 \right)_{i,k} \right] X_k = 0, \]
	and
	\[ \sum_i \left( P_3 Y \right)_i = \sum_i \sum_k \left[ \left( P_3 \right)_{i,k} Y_k \right] = \sum_k \left[ \sum_i \left( P_3 \right)_{i,k} \right] Y_k = C \sum_k Y_k. \]
	We infer that 
	\[ \sum_k Y_k = \sum_i \left( P_3^{-1} Q_3 X \right)_i = \sum_i \left( D_{xxx} X \right)_i = 0. \]
	We therefore obtain ${\displaystyle \sum_i \left( I_y \otimes D_{xxx} \frac{U_{n+1}+U_{n}}{2} \right)_i = 0 }$.
	
	In the same way, we have ${\displaystyle \sum_i \left[ \left( I_y \otimes D_x \right) X \right]_i = 0}$ and ${\displaystyle \sum_i \left[ \left( D_{yy} \otimes I_x \right) X \right]_i = 0}$ for any vector $X$. Since 
	\[ {\displaystyle \left( I_y \otimes \left( \overline{D_x} \right)^{-1} \right) \left( D_{yy} \otimes I_x \right) = D_{yy} \otimes \left( \overline{D_x} \right)^{-1} = \left( D_{yy} \otimes I_x \right) \left( I_y \otimes \left( \overline{D_x} \right)^{-1} \right)},\]
	\eqref{schema_SS_SC} implies ${\displaystyle \sum_{i} \left( U_{n+1} \right)_i = \sum_{i} \left( U_{n} \right)_i }$.
\end{proof}

The conservation of the $L^2$-norm 
\[ \left\langle I_y \otimes D_x \left( \frac{U_{n+1} + U_n}{2} \right)^{p+1}, \ \frac{U_{n+1} + U_n}{2} 		\right\rangle = 0 \]
is more difficult to reach with high order schemes. Nevertheless, the precision of compact schemes will ensure a good conservation of the $L^2$-norm, as shown in the numerical simulations.

\begin{remark}\label{rem_syst_lin}
	In order to solve the linear system $AX=b$, where
	\[ A = I_y \otimes \left( I_x + \frac{\Delta t}{2} D_{xxx} \right) + \lambda \frac{\Delta t}{2} \left( D_{yy} \otimes \left( \overline{D_x} \right)^{-1} \right), \]
	we use a diagonal preconditioner and a GMRES method.
\end{remark}

\subsubsection{Fourier spectral method}	

Because of the definition of the antiderivative ${\displaystyle \partial_x^{-1} = \mathscr{F}^{-1} \left( \frac{1}{i\xi} \right)}$, the Fast Fourier Transform (FFT) is the most common way to solve the KP equation with periodic boundary conditions in space.
We define the Fourier modes $k_x = \dfrac{2\pi k}{L_x}$ and $k_y = \dfrac{2\pi l}{L_y}$ in the domain $[-L_x, L_x] \times [-L_y, L_y]$.
The Fourier transform applied to the KP equation,gives $\forall k \in \Z^*,\ l \in \Z$
\[ \frac{d \hat{u}}{\Delta t}(k,l,t) - i \left[ k_x^3 - \lambda \frac{k_y^2}{k_x} \right] \hat{u}(k,l,t) + \frac{ik_x}{p+1} \mathscr{F}\left( u^{p+1} \right)(k,l,t) = 0. \]
The mass zero constraint is translated as $\hat{u}(0,l,t)=0$.

In order to conserve the $L^2$-norm, a Sanz-Serna scheme is used again
\begin{equation}\label{fft1}
	\frac{\hat{u}_{n+1}-\hat{u}_{n}}{\Delta t} - i \left[ k_x^3 - \lambda \frac{k_y^2}{k_x} \right] \frac{\hat{u}_{n+1}+\hat{u}_{n}}{2} + i\frac{k_x}{p+1}\mathscr{F}\left( \left( \frac{u_n + u_{n+1}}{2} \right)^{p+1} \right) = 0. 
\end{equation}

With the following matrices (where $N_x$ and $N_y$ are powers of 2)
\begin{align*} 
	K_{\star} = & \frac{2\pi}{L_{\star}}\mbox{diag}\left( 0, \ 1, \ \cdots, \ \frac{N_{\star}}{2}-1, \ -\frac{N_{\star}}{2}, \ \cdots, \ -1 \right), \quad \star = x \mbox{~or~} y,  \\
	\left( \overline{K_x} \right)^{-1} = & \frac{2\pi}{L_x}\mbox{diag}\left( 0, \ 1, \ \cdots, \ \frac{1}{\frac{N_x}{2}-1}, \ -\frac{2}{N_x}, \ \cdots, \ -1 \right),  \\
\end{align*}
we solve thanks to a Picard fixed-point method
\begin{equation}\label{fft2}
	\frac{\hat{U}_{n+1}-\hat{U}_{n}}{\Delta t} - i \left[ I_y \otimes K_x^3 - \lambda K_y^2 \otimes \left( \overline{K_x} \right)^{-1} \right] \frac{\hat{U}_{n+1}+\hat{U}_{n}}{2} + i\left( I_y \otimes \left( \overline{K_x} \right)^{-1} \right) \frac{1}{p+1}\mathscr{F}\left( \left( \frac{U_{n+1}+U_{n}}{2} \right)^{p+1} \right)= 0.
\end{equation}
		
It is worth to notice that the matrix $K_x$ is not invertible. 
The average in $x$ beeing zero, we can define the matrix $\left( \overline{K_x} \right)^{-1}$.

\subsubsection{Mixed spectral-compact schemes}	 

The wave, solution of the KP equation, propagates mainly in the $x$-direction. For realistic configurations or less consuming computational, it is natural to consider a large domain in the $x$-direction and to restrict the spatial domain in the $y$-direction. In this case, we combine FFT in the $x$-direction and compact schemes with Dirichlet and/or Neumann condition in $y$ as follows

\begin{multline*}
	\frac{\hat{U}_{n+1} - \hat{U}_n}{\Delta t} - i \left( I_y \otimes K^3 \right) \frac{\hat{U}_{n+1} + \hat{U}_n}{2} \\
	+ \frac{i}{p+1} \left(I_y \otimes K \right) \mathscr{F} \left( \left( \frac{U_n + U_{n+1}}{2} \right)^{p+1} \right) - i\lambda \left( D_{yy} \otimes K^{-1} \right) \frac{\hat{U}_{n+1} + \hat{U}_n}{2} = 0. 
\end{multline*}
Gathering $\hat{U}^{n+1}$ and $\hat{U}^n$ gives

\begin{multline*}
	\left[ I_d - i\frac{\Delta t}{2} \left( I_y \otimes K^3 + \lambda \Delta t D_{yy} \otimes K^{-1} \right) \right] \hat{U}_{n+1} = \\
	\left[ I_d + i\frac{\Delta t}{2} \left( I_y \otimes K^3 + \lambda \Delta t D_{yy} \otimes K^{-1} \right) \right] \hat{U}_{n} + i\frac{\Delta t}{p+1} \left(I_y \otimes K \right) \mathscr{F} \left( \left( \frac{U_n + U_{n+1}}{2} \right)^{p+1} \right).
\end{multline*}

%\begin{remark}
%	The scheme conserves the mean in $x$. 
%%	Indeed, it is obtained for $k=0$, which means 
%%	\[ \frac{\hat{U}_{n+1} - \hat{U}_n}{\Delta t} = 0. \]
%	Consequently, for this scheme, the mean in $x$ can be different to zero but we lose the physical meaning of the equation.
%\end{remark}

\section{Numerical results}
	After establishing the convergence, the schemes are tested and validated according to three well known behaviors of the KP equations, namely, the perturbations of Zaitsev soliton, transverse instabilities and blow-up in finite time \cite{KS10}. The codes are written in python.

\subsection{Order estimates and comparison}

We consider the Zaitsev travelling waves as exact solution the KP-I equation with $p=1$. It is defined by
\begin{equation}\label{Zaitsev_soliton}
	\Psi(x,y,t) = 12 \alpha^2 \frac{1 - \beta \cosh (\alpha x - \omega t) \cos (\delta y)}{\left( \cosh (\alpha x - \omega t) - \beta \cos (\delta y) \right)^2}, 
\end{equation}
where
\[ \beta = \sqrt{\frac{\delta^2 - 3 \alpha ^4}{\delta ^2}}, \ 3\alpha^4 < \delta^2, \ \omega = \frac{\delta^2 + \alpha^4}{\alpha}, \ c = \frac{\omega}{\alpha}. \]
The computational domain is $\Omega = [-L_x;L_x] \times [-L_y;L_y] = [-89.6;89.6] \times [-21;21]$ and we take $\alpha = 0.0174$, $\delta = \frac{\pi}{L_y} = \frac{\pi}{21}$ and $c = 0.76$ \cite{Ha01}, and the time step $dt = 10^{-4}$. The initial datum is represented in Figure \ref{zaitsev_coupe}.

\begin{figure}[H]
	\centering
	\includegraphics[scale=0.4]{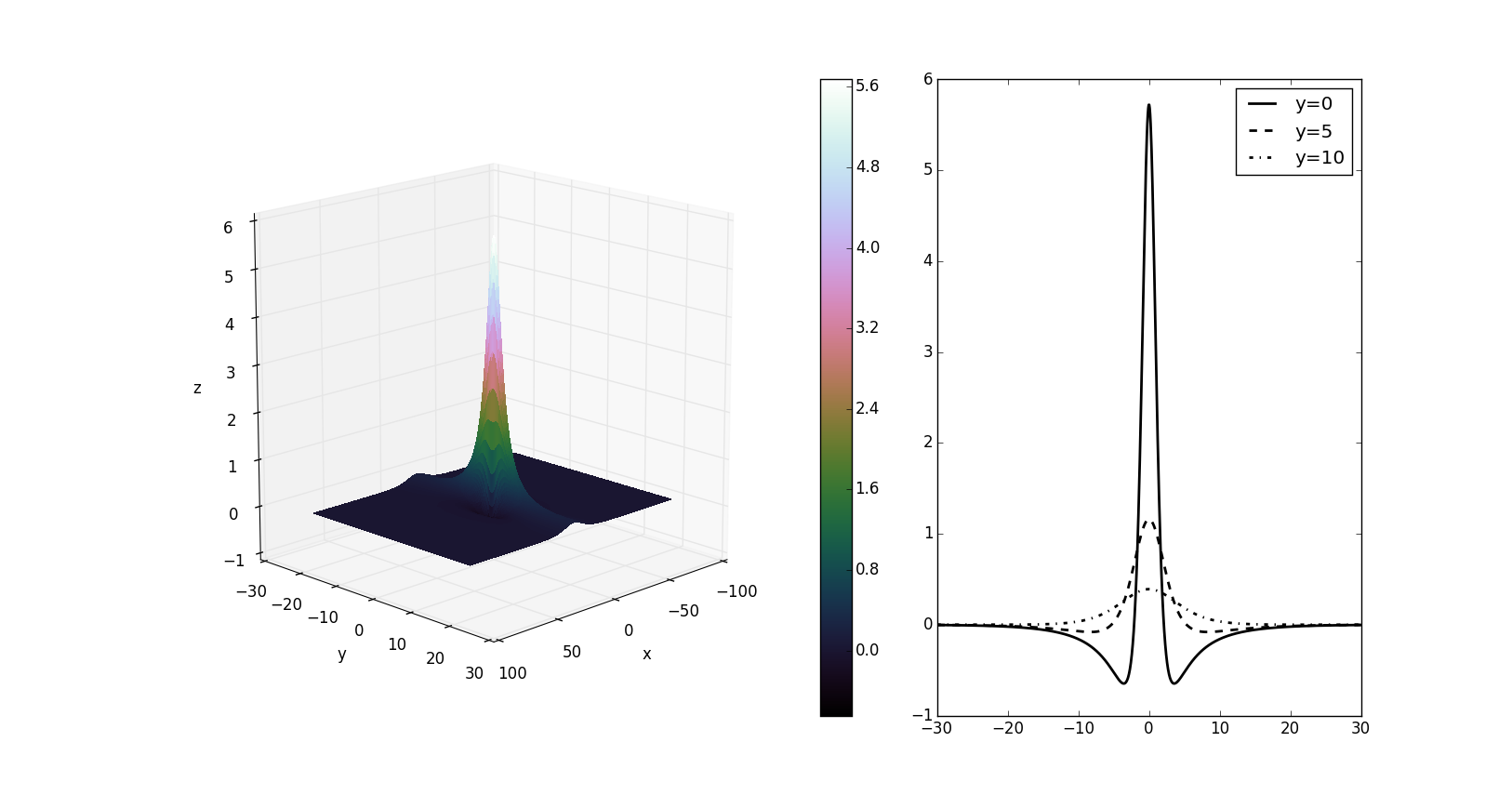}
	\caption{At left, 3-dimensional representation of the Zaitsev travelling wave. At right, 1-dimensional projection for $y=$0 (-), 5 (- -), 10 (-$\cdot$-).}
	\label{zaitsev_coupe}
\end{figure}

We compare the $L^2$-norm error obtained by the schemes presented in section \ref{sec:num_scheme} with varying space discretization. The results are listed in Table \ref{tab_norm_L2}. It is not surprising that the classical Fourier scheme gives the best result. Nevertheless, the error given  by compact scheme is well decreasing and the results are quite good from the fourth order.

\begin{center}
\tablehead{\hline Method & $N_x$ & $N_y$ & $L^2$-norm error  \\}
\tabletail{\hline}
\bottomcaption{\label{tab_norm_L2} $L^2$-norm error for different settings}
\begin{supertabular}{|c|c|c|c|}
	\hline
	full FFT & $2^9$ & 200 & $7.74.10^{-7}$\\
	\hline
	full 2$^{nd}$ order compact schemes & 601 & 160 & $9.94.10^{-2}$\\
	\hline
	full 4$^{th}$ order compact schemes & 601 & 160 & $1.45.10^{-2}$\\
	\hline
	full 6$^{th}$ order compact schemes & 601 & 160 & $5.03.10^{-4}$\\
	\hline
	\multirow{3}{*}{FFT in $x$ and 2$^{nd}$ order compact scheme in $y$} & \multirow{3}{*}{$2^9$} & 100 & $5.35.10^{-2}$\\
	\cline{3-4}
	& & 150 & $2.37.10^{-2}$\\
	\cline{3-4}
	& & 200 & $1.33.10^{-2}$\\
	\hline
	\multirow{3}{*}{FFT in $x$ and 4$^{th}$ order compact scheme in $y$} & \multirow{3}{*}{$2^9$} & 100 & $9.38.10^{-4}$\\
	\cline{3-4}
	& & 150 & $1.92.10^{-4}$\\
	\cline{3-4}
	& & 200 & $6.04.10^{-5}$\\
	\hline
	\multirow{3}{*}{FFT in $x$ and 6$^{th}$ order compact scheme in $y$} & \multirow{3}{*}{$2^9$} & 100 & $4.59.10^{-5}$\\
	\cline{3-4}
	& & 150 & $3.98.10^{-6}$\\
	\cline{3-4}
	& & 200 & $1.03.10^{-6}$\\
\end{supertabular}
\end{center}

Figure \ref{pente_mixte} presents the numerical order obtained from mixed schemes with respect to $N_y$.

% Ici mix fft+SC
\begin{figure}[H]
	\centering
	\includegraphics[scale=0.4]{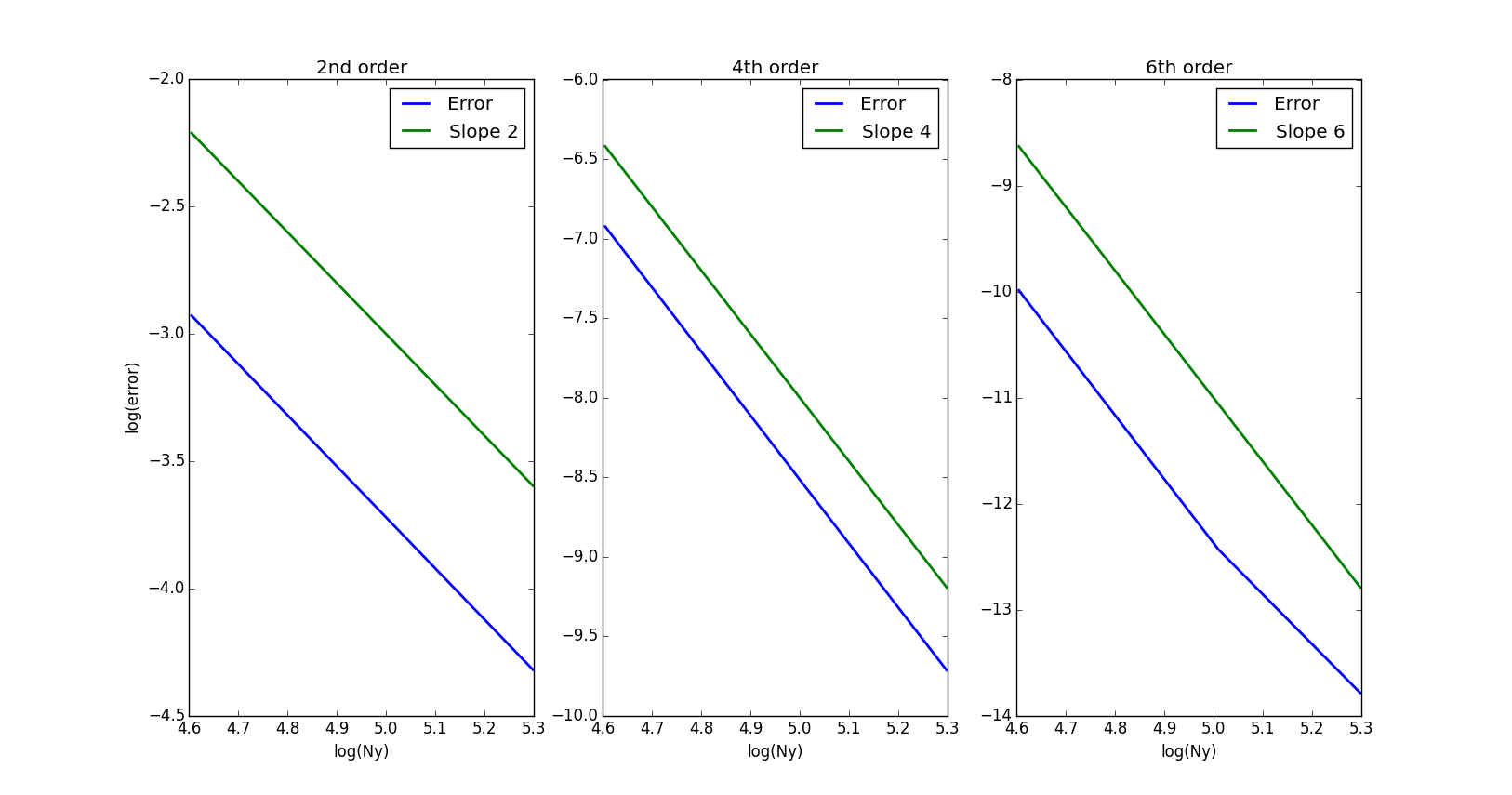}
	\caption{$L^2$-norm error with respect to the discretization step computed from mixed schemes. Here we have mixed schemes.}
	\label{pente_mixte}
\end{figure}

To verify the mass zero constraint, we choose as initial datum (Figure \ref{gaussienne_init_norm})
\[ u_0(x,y) = A(1-2sx^2) {\rm e}^{-sx^2-ty^2} \mbox{~with~} s=0.25, \ t=7.5 \mbox{~and~} A=5. \]
Here, the domain is $[-25, 25] \times [-5, 5]$ and $dt=10^{-4}$ and other parameters are sum up in Table \ref{settings_gaussienne}.
\begin{center}
\tablehead{\hline Method & $N_x$ & $N_y$\\}
\tabletail{\hline}
\bottomcaption{\label{settings_gaussienne} Parameters for a Gaussian as initial datum.}
\begin{supertabular}{|c|c|c|}
	\hline
	full 4$^{th}$ order compact scheme & 501 & 100 \\
	\hline
	FFT in $x$ and 4$^{th}$ order compact scheme in $y$ & $2^9$ & 100 \\
	\hline
	full FFT & $2^9$ & $2^7$ \\
\end{supertabular}
\end{center}

\begin{figure}[H]
	\centering
	\includegraphics[scale=0.43]{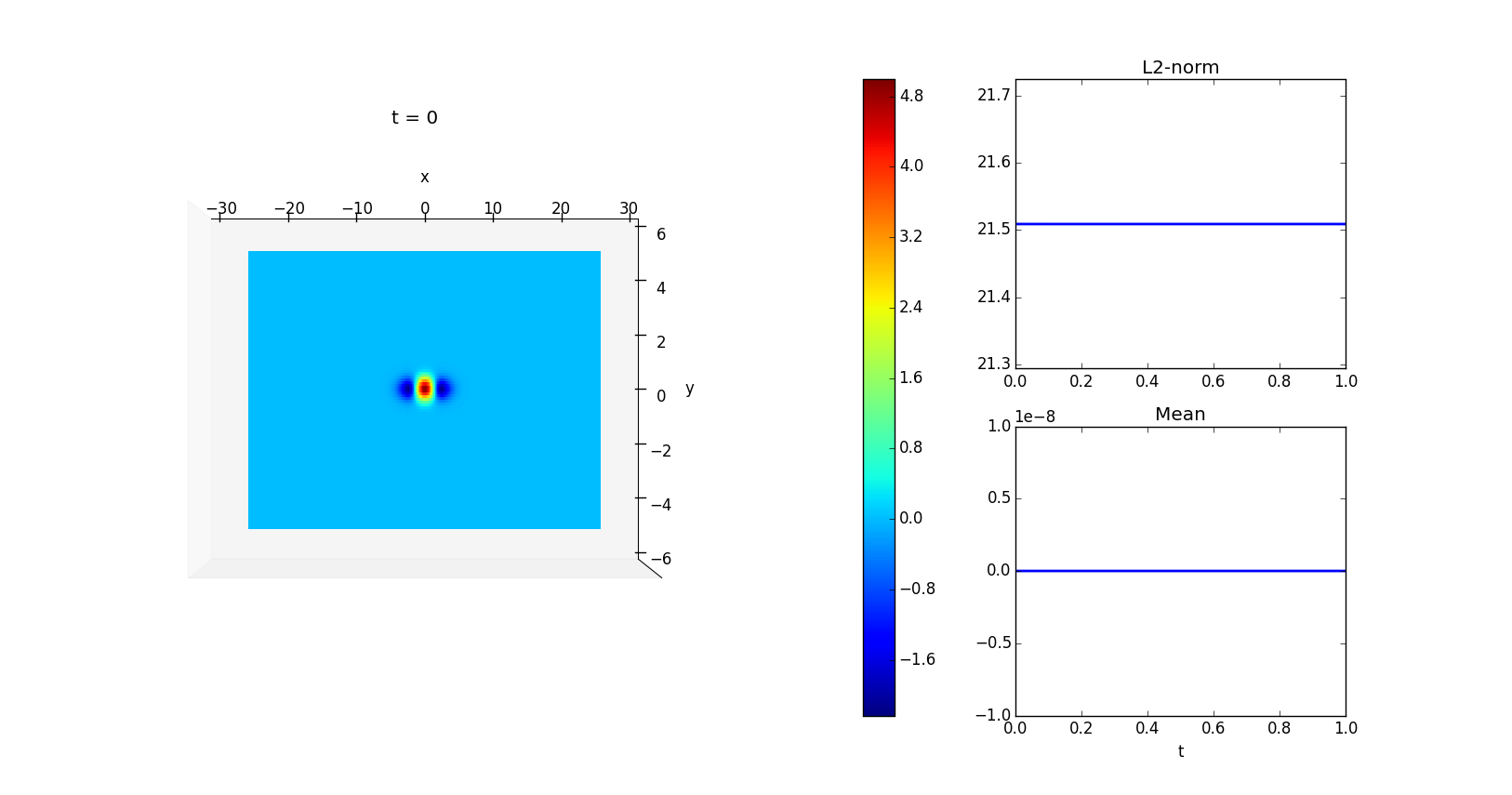}
	\caption{At left, the Gaussian initial datum with $A=5$, $s=0.25$ and $t=7.5$. At right, the time evolution of the $L^2$-norm and mean computed with the 4$^{th}$ order compact scheme.}
	\label{gaussienne_init_norm}
\end{figure}

\begin{figure}[H]
	\centering
	\includegraphics[scale=0.43]{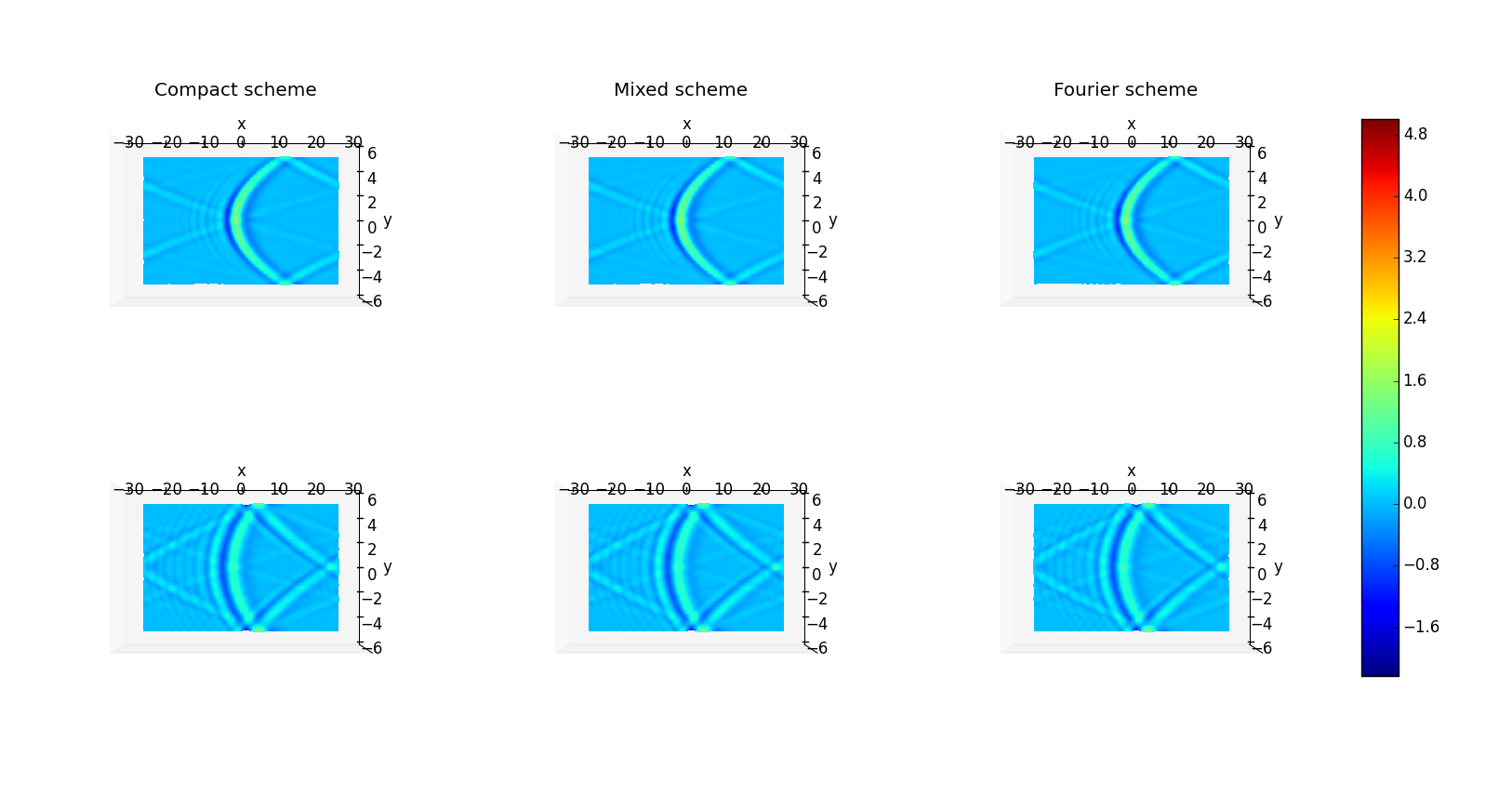}
	\caption{Solutions of KP-I equations at time $t=0.5$ (first row) and $t=1$ (second row). The first column denotes 4$^{th}$ order compact scheme, the second column is 4$^{th}$ order mixed scheme and third column represents the FFT.}
\end{figure}

\subsection{Perturbations of the Zaitsev soliton}

The Zaitsev soliton \eqref{Zaitsev_soliton} is perturbed with a Gaussian function. More precisely, the simulation starts from the following initial datum 
\[ u_0(x,y) = \Psi \left( x+\frac{L_x}{2},y,0 \right) + 6 \left( x+\frac{L_x}{2} \right) {\rm e}^{-\left( x+\frac{L_x}{2} \right)^2 - y^2}. \]
Here we take $\alpha=1$, $\beta=0.5$ and $\delta = 3$ in the domain $[-25,25] \times [-5,5]$. We present in Figure \ref{fig_pzaitsev_sol} the simulation performed with 4$^{th}$ order compact schemes with $N_x=501$, $N_y=100$ and $dt=10^{-4}$. The $L^2$-norm, the mean are well conserved and the $L^{\infty}$-norm has a consistent behaviour (Figure \ref{fig_pzaitsev_norms}). As expected \cite{KS10}, Figure \ref{fig_pzaitsev_sol} shows that peaks are appearing and the Zaitsev soliton develops into a lump.

\begin{figure}[H]
	\centering
	\includegraphics[scale=0.4]{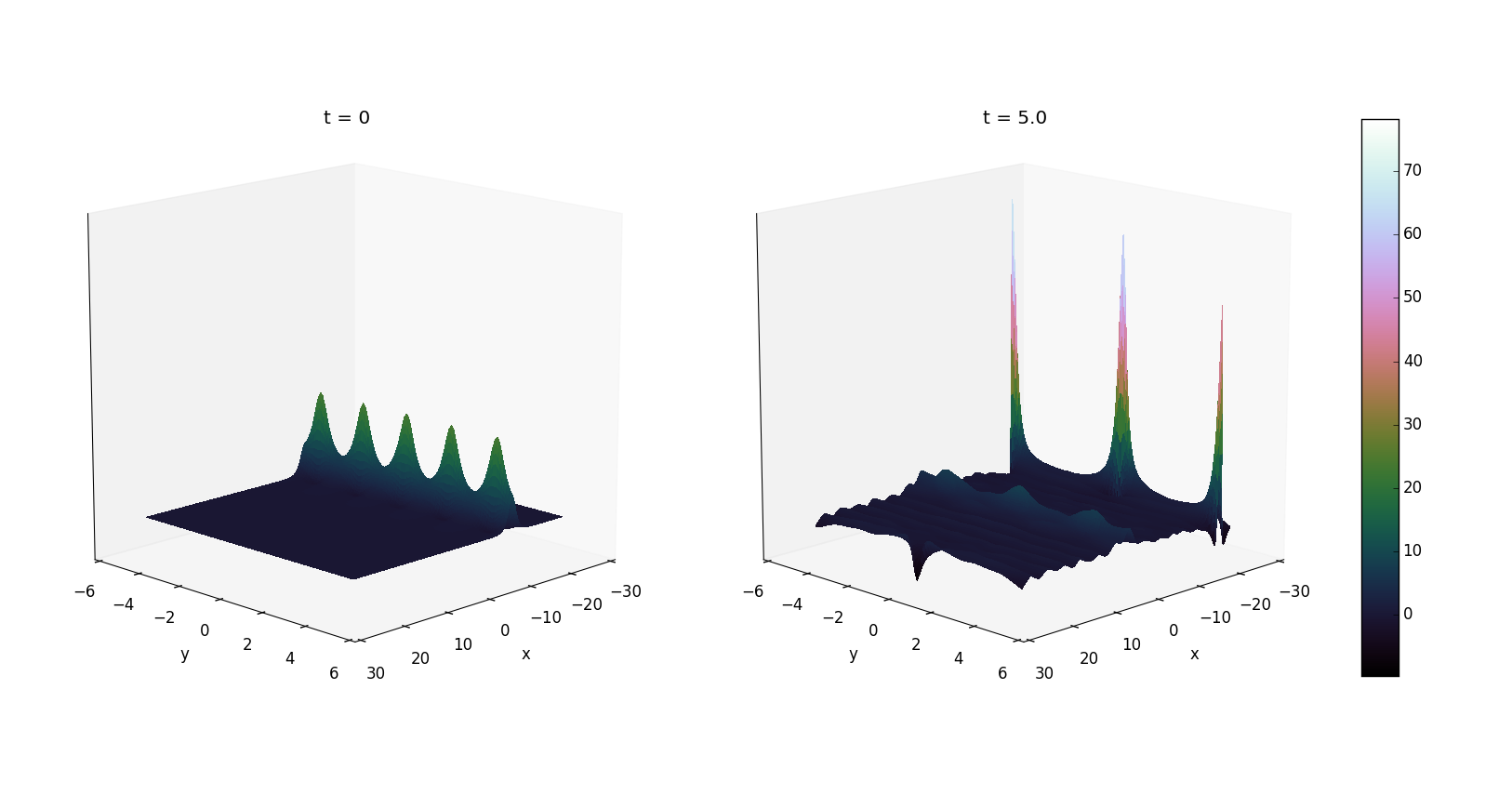}
	\caption{At left, perturbed Zaitsev soliton as initial datum. At right, solution of the KP-I equation with $p=1$ at time $t=5$.}
	\label{fig_pzaitsev_sol}
\end{figure}

\begin{figure}[H]
	\centering
	\includegraphics[scale=0.4]{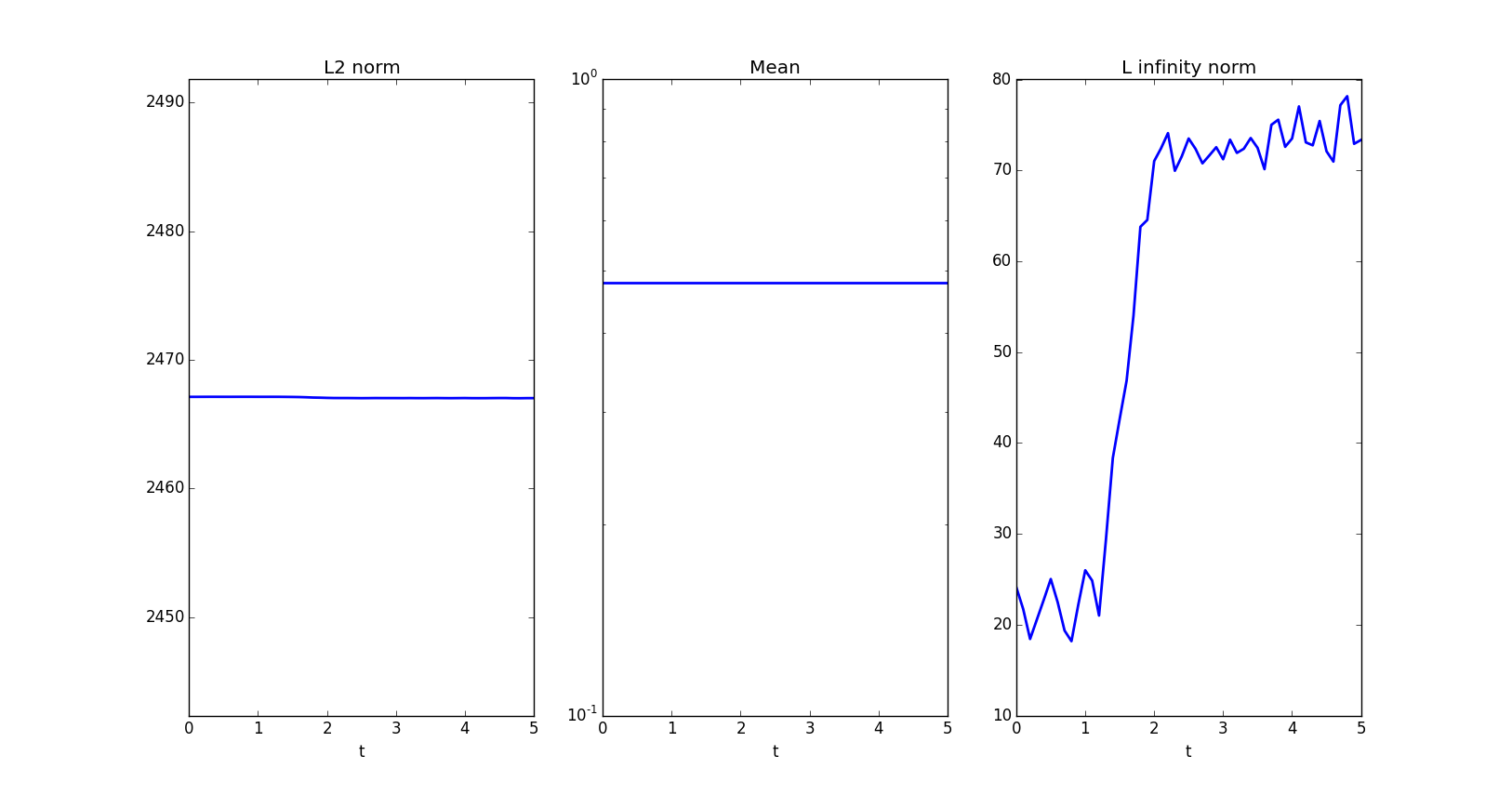}
	\caption{$L^2$-norm, mean and $L^{\infty}$-norm with respect to time of the solution of the KP-I equation starting from perturbed Zaitsev.}
	\label{fig_pzaitsev_norms}
\end{figure}

\subsection{Transverse instabilities of the line-soliton}

We investigate the transverse instabilities \cite{APS97, RT08, RT09, RT10, RT11, RT12} of the Korteweg-de Vries soliton. We consider the perturbed line-soliton
\[ u_0(x,y) = 12 \sech^2 \left( x + 0.4 \cos \left( \frac{2y}{Ly} \right) \right). \]
In Figure \ref{fig_pline_sol}, we use 4$^{th}$ order compact schemes with $N_x=501$, $N_y=100$ and $dt=10^{-4}$. In Figure \ref{fig_pline_norms}, we can observe that the norms are well conserved and that the $L^{\infty}$-norm is increasing, translating that the line-soliton is unstable as shown  in Figure \ref{fig_pline_sol} \cite{Ha01, HMM09, KS10}.

\begin{figure}[H]
	\centering
	\includegraphics[scale=0.4]{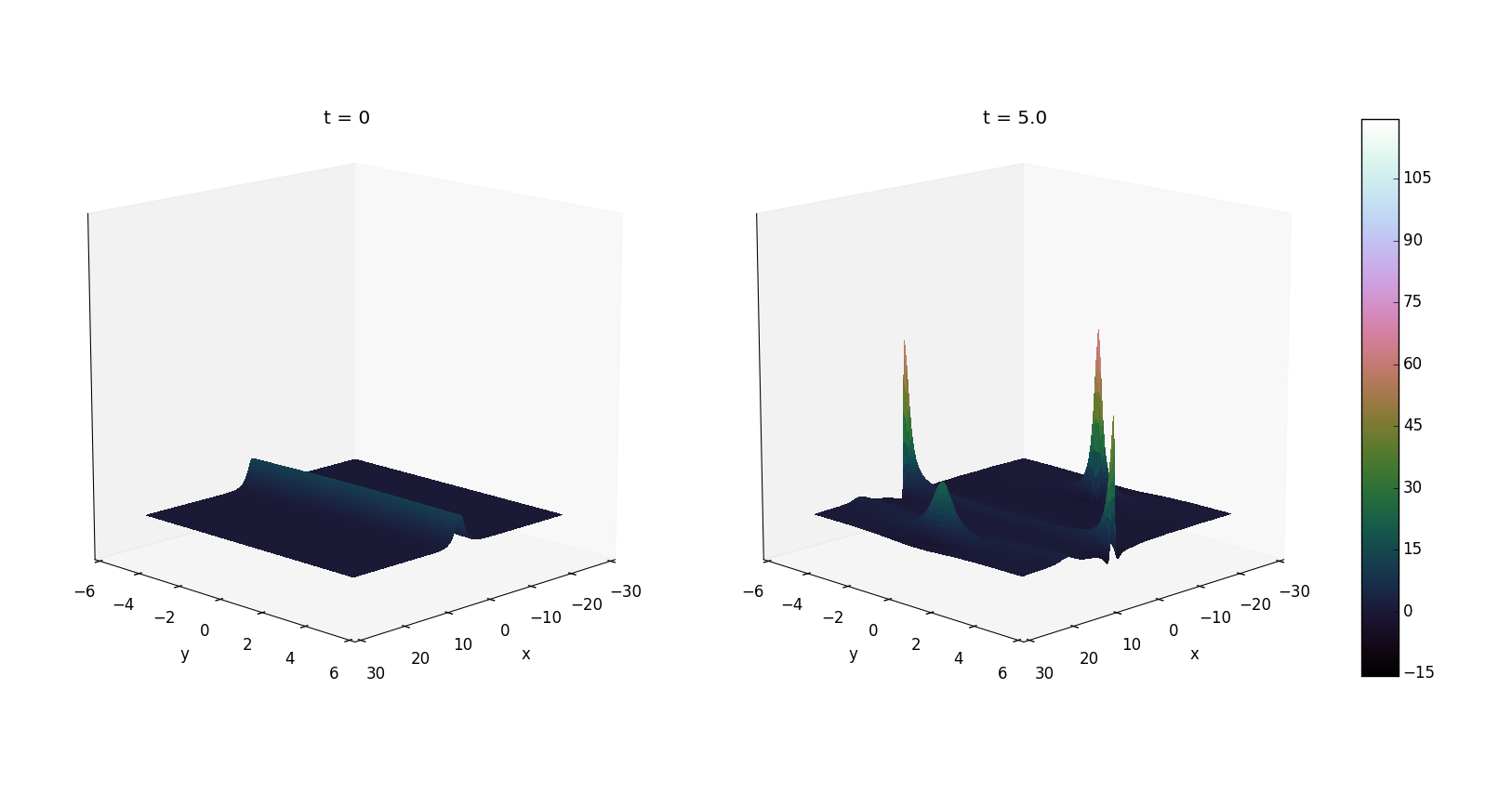}
	\caption{At left, perturbed KdV soliton as initial datum. At right, solution of the KP-I equation with $p=1$ at time $t=5$.}
	\label{fig_pline_sol}
\end{figure}

\begin{figure}[H]
	\centering
	\includegraphics[scale=0.4]{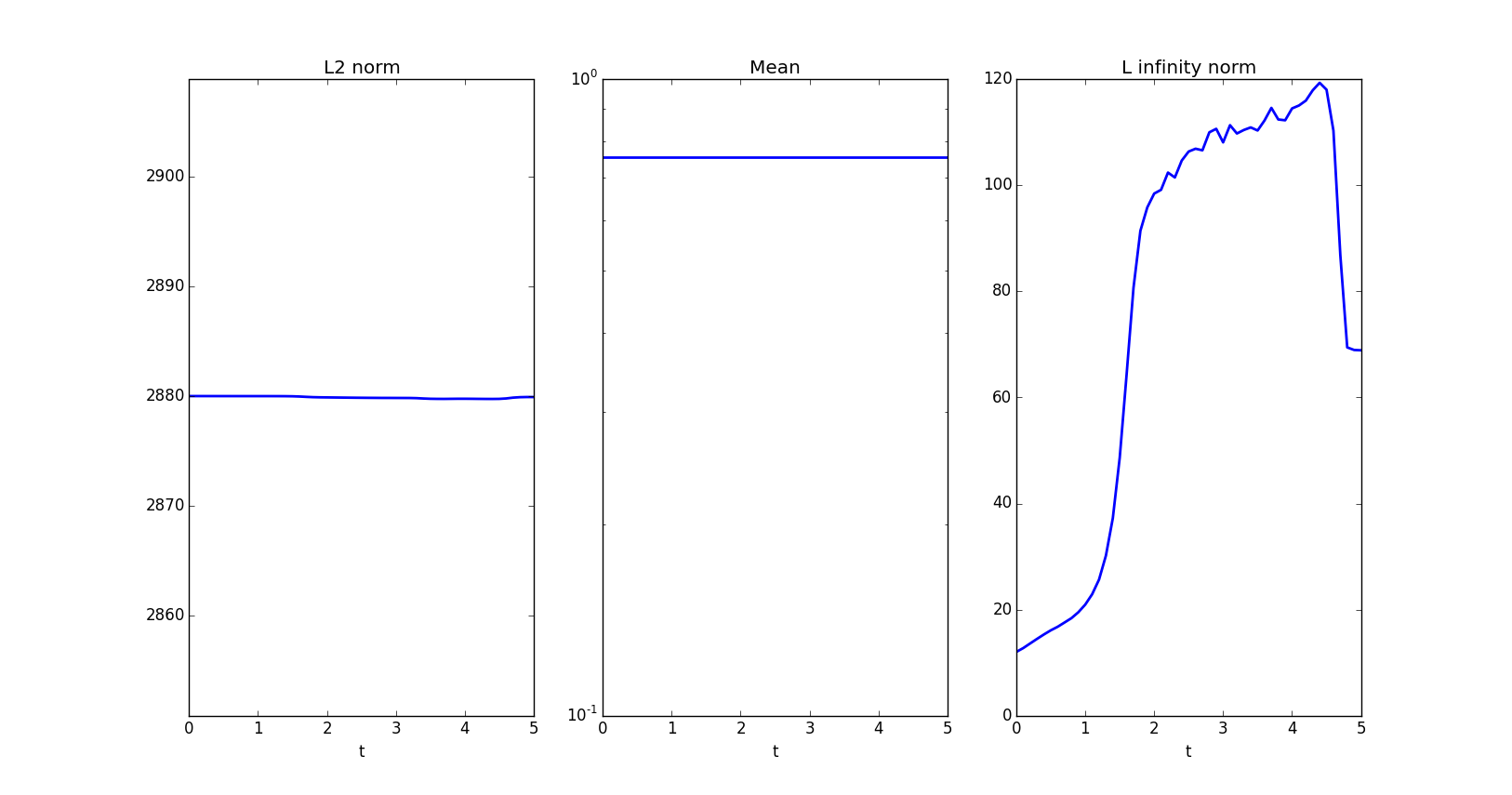}
	\caption{$L^2$-norm, mean and $L^{\infty}$-norm with initial datum equal to a perturbed KdV soliton (cf Figure \ref{fig_pline_sol})}
	\label{fig_pline_norms}
\end{figure}

\subsection{Blow-up in finite time}

To illustrate the blow-up in finite time of the KP-I equation when $p=2$, we consider the initial datum
\[ u_0(x,y) = 3 \partial_{xx} {\rm e}^{-(x^2+y^2)}. \]
We use 6$^{th}$ order compact schemes in the domain $[-10,10]\times[-2.5,2.5]$ with $N_x=201$, $N_y=50$ and $dt=10^{-6}$. Figure \ref{fig_blowup_sol} shows that the minimum blows up. We see in Figure \ref{fig_blowup_norms} that the $L^2$-norm and the mean are conserved, whereas the $L^{\infty}$-norm becomes singular.

\begin{figure}[H]
	\centering
	\includegraphics[scale=0.4]{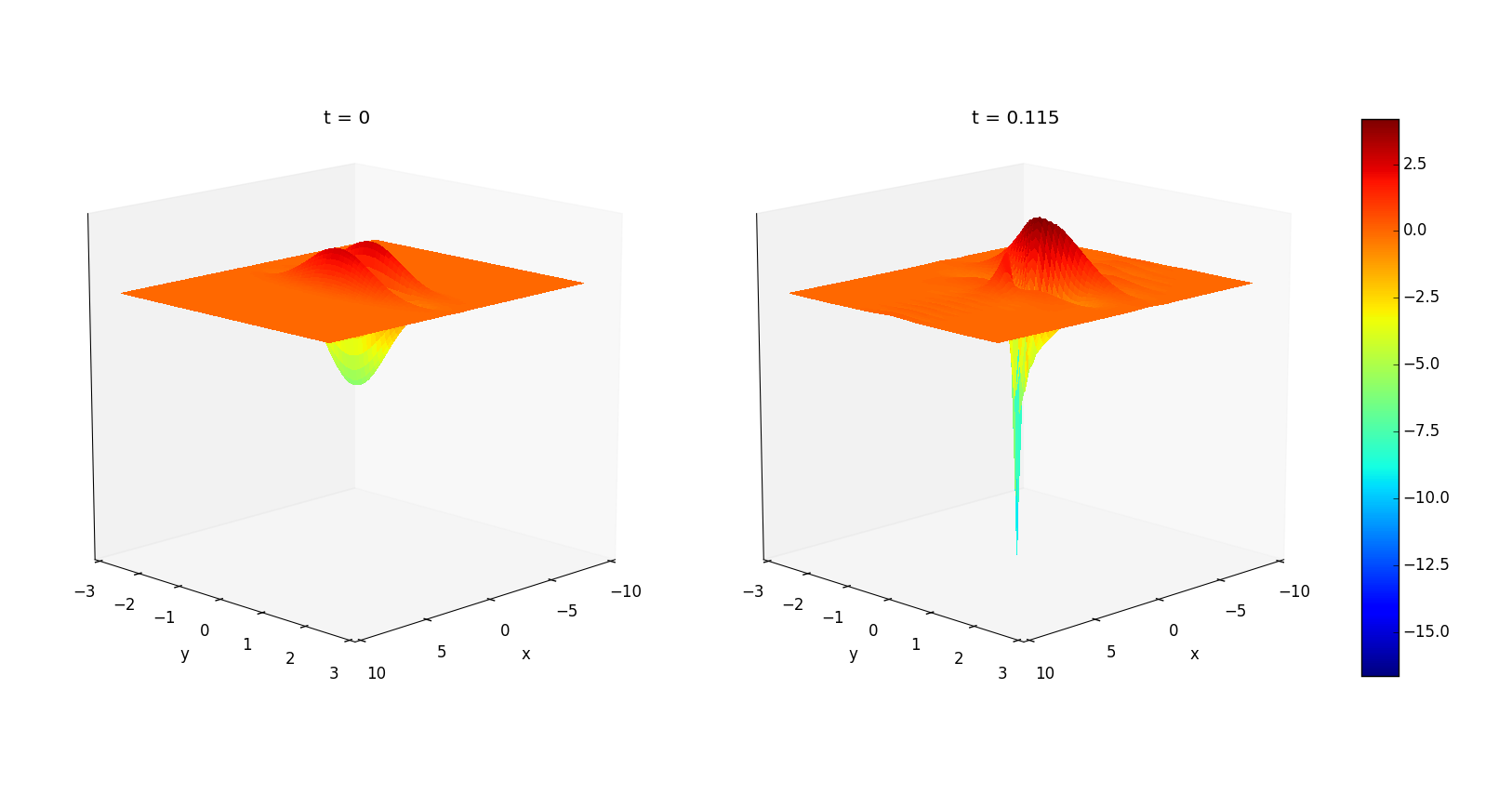}
	\caption{At left, second derivative of a Gaussian as initial datum. At right, blow-up of the solution of KP-I with $p=2$ at time $t=0.115$.}
	\label{fig_blowup_sol}
\end{figure}

\begin{figure}[H]
	\centering
	\includegraphics[scale=0.4]{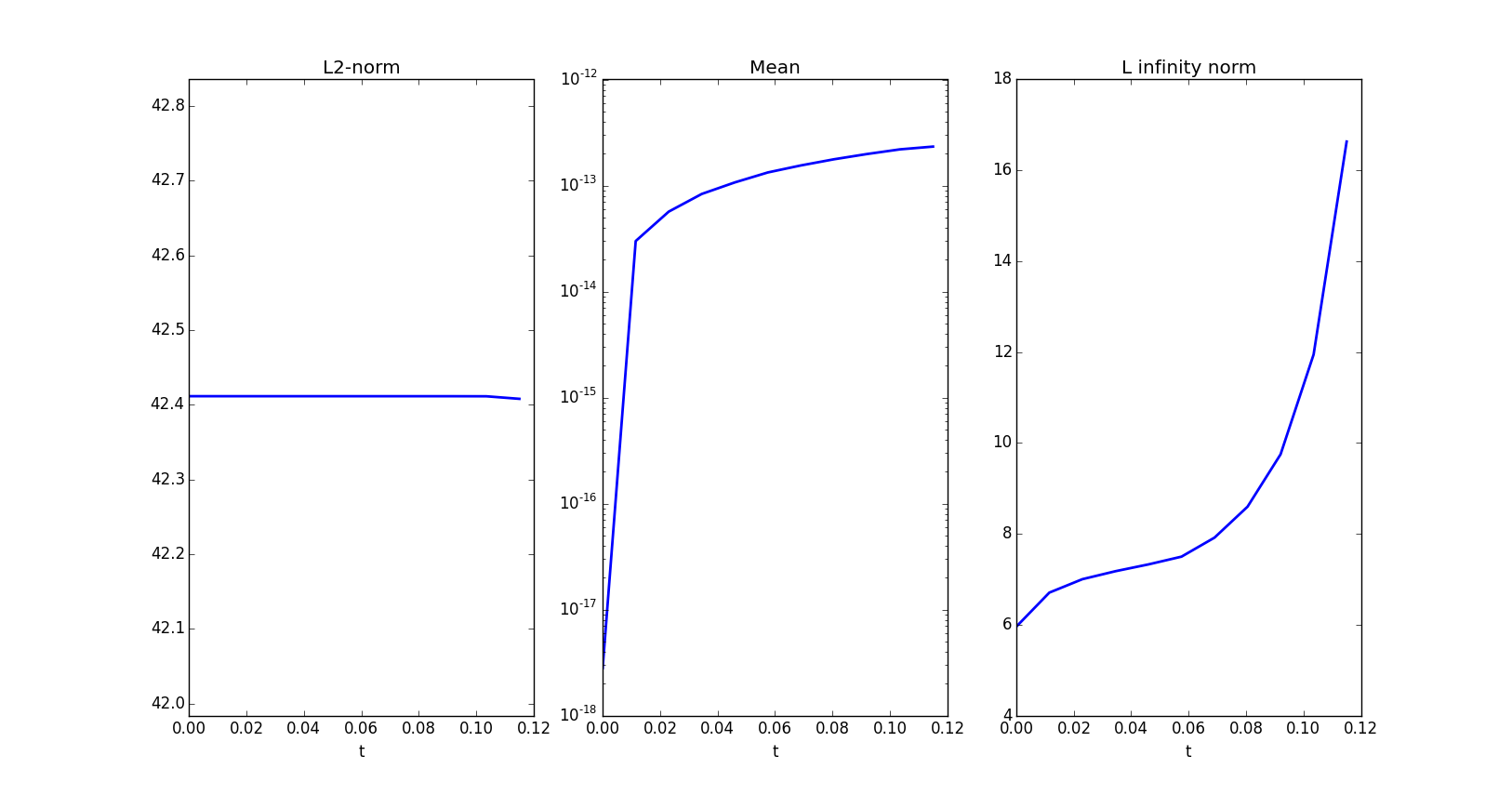}
	\caption{Evolution with time of the $L^2$-norm, mean and $L^{\infty}$-norm with the second derivative of a Gaussian as initial datum.}
	\label{fig_blowup_norms}
\end{figure}

\section{Conclusion}
	In this paper, we propose new compact schemes to solve the Kadomtsev-Petviashvili equations. 
We planned to do the simulations with periodic boundary conditions, these schemes can be easily adapted to other boundary conditions. This allows to find the expected results in terms of error, blow-up and transverse instabilities with high precision schemes.

\end{document}